\documentclass[12pt]{article}
\usepackage[a4paper, total={17cm,25cm}]{geometry}

\usepackage{amsmath,amssymb,amsfonts}
\usepackage[alphabetic,y2k,initials]{amsrefs}

\usepackage{color, graphics}
\usepackage{float}

\usepackage{tikz}
\usepackage{pgfplots,tikz-3dplot}

\usepackage{pdfsync}
\usepackage{hyperref}
\usepackage{multirow}

\usepackage{longtable}

\usepackage{calrsfs}
\DeclareMathAlphabet{\pazocal}{OMS}{zplm}{m}{n}
\newcommand{\LL}{\mathcal{L}}

\newcommand{\Q}{\pazocal{Q}}
\newcommand{\E}{\pazocal{E}}

\newcommand{\Pol}{\mathcal{P}}

\newcommand{\Curve}{\mathcal{C}}

\newcommand{\refeq}{\eqref}

\def\rank{\mathrm{rank}}
\def\sign{\mathrm{sign\,}}

\def\dist{\mathrm{dist}}

\newtheorem{theorem}{Theorem}[section]
\newtheorem{proposition}[theorem]{Proposition}
\newtheorem{corollary}[theorem]{Corollary}
\newtheorem{lemma}[theorem]{Lemma}
\newtheorem{remark}[theorem]{Remark}

\newenvironment{proof}[1][Proof]{\noindent\textit{#1.} }{\hfill$\Box$\newline\medskip}

\numberwithin{equation}{section}

\title{Periodic trajectories of ellipsoidal billiards in the $3$-dimensional Minkowski space}

\usepackage{authblk}
\author[1,3]{Vladimir Dragovi\'c}
\author[2,3]{Milena Radnovi\'c}
\affil[1]{\textsc{The University of Texas at Dallas, Department of Mathematical Sciences}}
\affil[2]{\textsc{The University of Sydney, School of Mathematics and Statistics}}
\affil[3]{\textsc{Mathematical Institute SANU, Belgrade}}
\affil[ ]{\texttt{vladimir.dragovic@utdallas.edu, milena.radnovic@sydney.edu.au}}

\date{}

\begin{document}

\maketitle

\begin{center}
\emph{Dedicated to Professor Nalini Joshi on the occasion of her anniversary.}
\end{center}

\begin{abstract}
In this paper, we give detailed analysis and description of periodic trajectories of the billiard system within an ellipsoid in the $3$-dimensional Minkowski space, taking into account all possibilities for the caustics.
The conditions for periodicity are derived in algebro-geometric, analytic, and polynomial form.
\end{abstract}

\tableofcontents

\section{Introduction}

Discrete integrable systems occupy an important part of the scientific activity and legacy of Professor Nalini Joshi.
There are several recent monographs related to discrete integrability as intensively developed field of pure and applied mathematics (see \cite{DuistermaatBOOK, BS2008book, HJN2016, Joshi2019book}).
Integrable billiards (see \cite{KozTrBIL, DragRadn2011book}) form an important class of discrete integrable systems. This paper is devoted  to integrable billiards in the $3$-dimensional Minkowski space, merging two lines of our previous studies.

We will derive here the periodicity conditions for such billiards in different forms: in algebro-geometric terms and in terms of polynomial functional equations.
More about extremal polynomials and related Pell's equations one can find in \cite{Akh4, Bogatyrev2012} and references therein.
Following \cite{BirkM1962, KhTab2009}, we  introduced notions of relativistic quadrics and applied them to billiards in the pseudo-Euclidean spaces in \cite {DragRadn2012adv}. 
In a more recent paper \cite{DragRadn2019cmp}, we established a fundamental relationship between periodic integrable billiards in the Euclidean spaces of arbitrary dimension and extremal polynomials and Pell's equations. We applied these ideas in more detail in the basic, planar cases in  \cite{DragRadn2019rcd} for the Euclidean metrics and in \cite{ADR2019rcd} for the Minkowski metric. 
In this work, we deal with the three-dimensional Minkowski space, as  a  gateway to the study of billiards in higher-dimensional pseudo-Euclidean spaces.
The results of this paper provide the solution to a known open problem, Problem 5.2 from \cite{GKT2007}, which is also Problem 7 from \cite{Tab2015}. 
See Remark \ref{rem:open-problem} for more detail.

The organization of the paper is as follows. Section \ref{sec:intro} introduces basic notation. Section \ref{sec:algebro} is devoted to algebro-geometric formulation of periodicity conditions, while Section \ref{sec:polynomial} derives the conditions of periodicity in terms of Pell's equations and related polynomial functional equations.

\section{Confocal families of quadrics and billiards}\label{sec:intro}

In this section, we recall necessary notions and propertes related to confocal families of quadrics and billiards within ellipsoids in the Minkowski space.
A more detailed account can be found in \cite{GKT2007,KhTab2009,DragRadn2012adv}.

\paragraph*{The Minkowski space $\mathbf{E}^{2,1}$}  is $\mathbf{R}^3$ with \emph{the Minkowski scalar product}: $\langle X,Y\rangle=X_1Y_1+X_2Y_2-X_3Y_3$.

\emph{The Minkowski distance} between points $X$, $Y$ is
$
\dist(X,Y)=\sqrt{\langle{X-Y,X-Y}\rangle}.
$
Since the scalar product can be negative, notice that the Minkowski distance can have imaginary values as well.

Let $\ell$ be a line in the Minkowski space, and $v$ its vector.
The line $\ell$ is called
\emph{space-like} if $\langle{v,v}\rangle>0$;
\emph{time-like} if $\langle{v,v}\rangle<0$;
and \emph{light-like} if $\langle{v,v}\rangle=0$.
Two vectors $x$, $y$ are \emph{orthogonal} in the Minkowski space if $\langle x,y \rangle=0$.
Note that a light-like vector is orthogonal to itself.

\paragraph*{Confocal families.}
Denote by
\begin{equation}\label{eq:ellipsoid}
\E :\
\frac{x_1^2}{a_1}+\frac{x_2^2}{a_2}+\frac{x_3^2}{a_3}=1,
\end{equation}
with $a_1>a_2$, $a_3>0$, an ellipsoid.
Let us remark that equation of any ellipsoid in the Minkowski
space can be brought into the canonical form (\ref{eq:ellipsoid})
using transformations that preserve the scalar product.

The family of quadrics confocal with $\E$ is:
\begin{equation}\label{eq:confocal}
\Q_{\lambda}\ :\
\frac{x_1^2}{a_1-\lambda} +\frac{x_2^2}{a_2-\lambda} +
\frac{x_3^2}{a_3+\lambda}=1,\qquad\lambda\in\mathbf{R}.
\end{equation}

The family (\ref{eq:confocal}) contains four geometrical types of quadrics:
\begin{itemize}
\item
$1$-sheeted hyperboloids oriented along $x_3$-axis, for $\lambda\in(-\infty,-a_3)$;

\item
ellipsoids, corresponding to $\lambda\in(-a_3,a_2)$;

\item
$1$-sheeted hyperboloids oriented along $x_2$-axis, for $\lambda\in(a_2,a_1)$;

\item
$2$-sheeted hyperboloids, for $\lambda\in(a,+\infty)$ -- these hyperboloids are oriented along $x_3$-axis.
\end{itemize}
In Figure \ref{fig:konfM3}, one non-degenerate quadric of each geometric type is shown.
\begin{figure}[h]
\centering
\includegraphics[width=5.87cm, height=8.38cm]{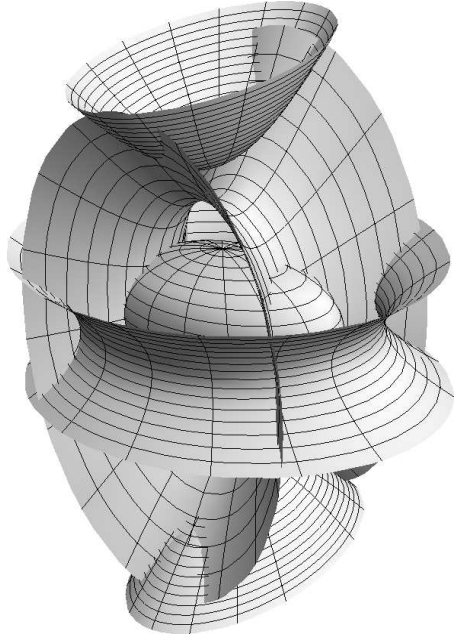}
\caption{Confocal quadrics in the three-dimensional Minkowski space.}\label{fig:konfM3}
\end{figure}
In addition, there are four degenerated quadrics: $\Q_{a_1}$, $\Q_{a_2}$, $\Q_{-a_3}$, $\Q_{\infty}$, that is planes $x_1=0$, $x_2=0$, $x_3=0$, and the plane at the infinity respectively.

The following theorem consists of a generalisation of the Chasles theorem to the Minkowski space and the additional conditions on the parameters of the quadrics touching a given line.
The corresponding generalisation of the Chasles theorem was first obtained in \cite{KhTab2009}, while the classification of the types of confocal quadrics touching a given line for the $3$-dimensional Minkowski space was considered in \cite{GKT2007}, regarding the geodesics on an ellipsoid, and in \cite{DragRadn2012adv} regarding billiards within ellipsoids in the pseudo-Euclidean space of arbitrary dimension.

\begin{theorem}\label{th:parametri.kaustike}
In the Minkowski space $\mathbf{E}^{2,1}$ consider a line intersecting ellipsoid $\E$ \refeq{eq:ellipsoid}.
Then this line is touching two quadrics from \refeq{eq:confocal}.
If we denote their parameters by $\gamma_1$, $\gamma_2$ and take:
	\begin{gather*}
	\{b_1,\ \dots,\ b_p,\ c_1,\ \dots,\ c_q\}=\{a_1,a_2,-a_3, \gamma_1,\gamma_{2}\},\\
	c_q\le\dots\le c_1<0<b_1\le\dots\le b_p,\quad p+q=5,
	\end{gather*}
	we will additionally have:
	\begin{itemize}
		\item
		if the line is space-like, then $p=3$, $q=2$, $a_1=b_3$, $\gamma_1\in\{b_{1},b_{2}\}$ for $1\le i\le k-1$, and $\gamma_{2}\in\{c_{1},c_{2}\}$;
		\item
		if the line is time-like, then $p=4$, $q=1$, $c_1=-a_3$, $\gamma_1\in\{b_{1},b_{2}\}$, $\gamma_2\in\{b_3,b_4\}$;
		\item
		if the line is light-like, then $p=4$, $q=1$, $b_4=\infty=\gamma_2$, $b_{3}=a_1$, $\gamma_1\in\{b_{1},b_{2}\}$, and $c_1=-a_3$.
	\end{itemize}
	Moreover, for each point on $\ell$ inside $\E$, there is exactly $3$ distinct quadrics from \refeq{eq:confocal} containing it.
	More precisely, there is exactly one parameter of these quadrics in each of the intervals:
	$$
	[c_1,0),\ (0,b_1],\ [b_2,b_3].
	$$
\end{theorem}

\begin{remark}\label{rem:intersection}
 Since $[c_1,b_1]\subset[-a_3,a_2]$ and $[b_2,b_3]\subset[a_2,a_1]$, any given point within $\E$ lies at the intersection of two ellipsoids and one $1$-sheeted hyperboloid oriented along $x_2$-axis.
\end{remark}

For each quadric, its \emph{tropic curves} are the sets of points where the induced metric on the tangent plane is degenerate.
An ellipsoid is divided by its tropic curves into three connected components: two of them are ``polar caps'' mutually symmetric with respect to the $x_1x_2$-plane, while the the third one is the ``equatorial'' annulus placed between them, see Figure \ref{fig:tropic}.
\begin{figure}[h]
\centering
\includegraphics[width=8.43cm, height=4.14cm]{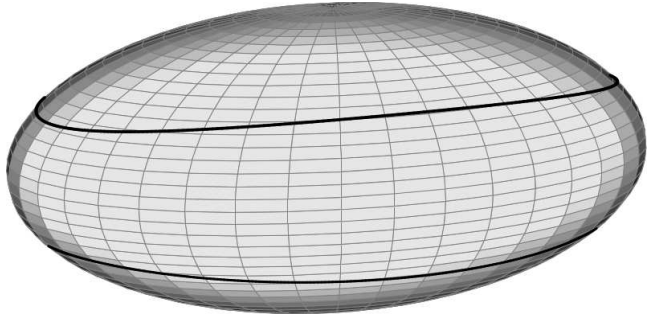}
\caption{Tropic curves on ellipsoid.}\label{fig:tropic}
\end{figure}
Notice that the induced metric within each ``polar cap'' is Reimannian, while it is Lorentzian in the annulus between the tropic curves \cite{GKT2007}.

\begin{remark}
From Theorem \ref{th:parametri.kaustike} and Remark \ref{rem:intersection}, we have that a given point $M$ within $\E$ lies at the intersection of three quadrics
$\Q_{\lambda_1}$, $\Q_{\lambda_2}$, $\Q_{\lambda_3}$, $\lambda_1<\lambda_2<\lambda_3$.
It is proved in \cite{DragRadn2012adv} that $M$ will belong to a ``polar cap'' of ellipsoid $\Q_{\lambda_1}$, and to the ``equatorial belt'' annulus of ellipsoid $\Q_{\lambda_2}$, with $-c\le\lambda_1<0<\lambda_2\le b$.

The triple $(\lambda_1,\lambda_2,\lambda_3)$ represents \emph{generalised elliptic coordinates} of $M$.
\end{remark}

\paragraph*{Billiards in the Minkowski space.}

Let $v$ be a vector and $\gamma$ a hyper-plane in the Minkowski space. Decompose vector $v$ into the sum
$v=a+n_{\gamma}$ of a vector $n_{\gamma}$ orthogonal to $\gamma$ and
$a$ belonging to $\gamma$. Then vector $v'=a-n_{\gamma}$ is
\emph{the billiard reflection} of $v$ on $\gamma$.
It is easy to see that then $v$ is also the billiard reflection of $v'$ with respect to
$\gamma$.


Note that $v=v'$ if $v$ is contained in $\gamma$ and $v'=-v$ if it
is orthogonal to $\gamma$. If $n_{\gamma}$ is light-like, which
means that it belongs to $\gamma$, then the reflection is not
defined.

Line $\ell'$ is a billiard reflection of $\ell$ off a smooth surface
$\pazocal{S}$ if their intersection point $\ell\cap\ell'$ belongs to $\pazocal{S}$
and the vectors of $\ell$, $\ell'$ are reflections of each other with respect to the tangent plane of $\pazocal{S}$ at this point.

\begin{remark}\label{remark:type}
It can be seen directly from the definition of reflection that the type of line is preserved by the billiard reflection.
Thus, the lines containing segments of a given billiard trajectory within $\pazocal{S}$ are all of the same type: they are all either space-like, time-like, or light-like.
\end{remark}

If $\pazocal{S}$ is an ellipsoid,
then it is possible to extend the reflection mapping to those points
where the tangent planes contain the orthogonal vectors. At such
points, a vector reflects into the opposite one, i.e.~$v'=-v$ and
$\ell'=\ell$. For the explanation, see \cite{KhTab2009}.
As follows from the explanation given there, it is natural to consider each such reflection as two reflections: one reflection off the ``polar cap'' and one off the ``equatorial belt''.

The following version of the Chasles' theorem holds for billiards within ellipsoids in the Minkowski space:

\begin{theorem}[\cite{KhTab2009}]
\label{th:chasles}
In the Minkowski space $\mathbf{E}^{2,1}$, consider a billiard trajectory within ellipsoid $\E$.
Then each segment of that trajectory is touching the same pair of quadrics confocal with $\E$.
\end{theorem}

The two quadrics from Theorem \ref{th:chasles} are called \emph{the caustics} of the trajectory.

\section{Periodic trajectories}\label{sec:algebro}

We will prove now the generalisation of the Poncelet theorem for the $3$-dimensional Minkowski space.
This proof is in the spirit of classical works of Jacobi and Darboux (see, for example \cite{JacobiGW,Darboux1870}), and also resembles to the proof of a Poncelet theorem for light-like geodesics on a quadric in the Minkowski space from \cite{GKT2007}.

\begin{theorem}\label{th:poncelet}
In the Minkowski space $\mathbf{E}^{2,1}$, consider an $n$-periodic billiard trajectory within ellipsoid $\E$.
Denote $n=m_1+n_1$, where $m_1$ is the total number of reflections off the ``polar caps'', and $n_1$ the number of reflections off the ``equatorial belt'' of $\E$ along the trajectory.
Then each billiard trajectory within $\E$ sharing the same pair of caustics is also $n$-periodic, with $m_1$ and $n_1$ reflections off the ``polar caps'' and ``equatorial belt'' respectively.
\end{theorem}
\begin{proof}
The differential equations in the elliptic coordinates of the lines touching two given quadrics $\Q_{\gamma_1}$ and $\Q_{\gamma_2}$ from \refeq{eq:confocal} are:
$$
\sum_{i=1}^3
\frac{d\lambda_i}{\sqrt{\Pol(\lambda_i)}}
=0,
\quad
\sum_{i=1}^3
\frac{\lambda_id\lambda_i}{\sqrt{\Pol(\lambda_i)}}
=0,
$$
with
\begin{equation}\label{eq:polynomial}
\Pol(x)=\varepsilon(a_1-x)(a_2-x)(a_3+x)(\gamma_1-x)(\gamma_2-x),
\quad
\varepsilon=\sign(\gamma_1\gamma_2).
\end{equation}
Introduce constants $b_1$, \dots, $b_p$, $c_1$, \dots, $c_q$ as in Theorem \ref{th:parametri.kaustike}.

Along billiard trajectory, each of the elliptic coordinates $\lambda_1$, $\lambda_2$, $\lambda_3$ takes values in segments $[c_1,0]$, $[0,b_1]$, $[b_2,b_3]$ respectively, with local extrema being only the end-points of the segments.

The value $\lambda_1=0$ corresponds to the reflection off a ``polar cap'' of $\E$, and $\lambda_2=0$ to the reflection off the ``equatorial belt''.
If $\lambda_1=\lambda_2=0$, then that corresponds to hitting the tropic curve, which will be counted as two reflections -- one off the ``polar cap'' and one off the ``equatorial belt''.
Whenever one of the elliptic coordinates takes value $a_1$, $a_2$, $-a_3$, the particle is crossing the coordinate plane $x_1=0$, $x_2=0$, $x_3=0$ respectively.
Values $\gamma_1$, $\gamma_2$ correspond to touching points with the caustics.

Similarly as in \cite{DarbouxSUR}, see also \cite{DragRadn2004}, integrating along the periodic trajectory gives:
\begin{equation}\label{eq:darboux}
m_1\int_0^{c_1}\frac{\lambda_1^k d\lambda_1}{\sqrt{\Pol(\lambda_1)}}
+
n_1\int_0^{b_1}\frac{\lambda_2^k d\lambda_2}{\sqrt{\Pol(\lambda_2)}}
-
n_2\int_{b_2}^{b_3}\frac{\lambda_3^k d\lambda_3}{\sqrt{\Pol(\lambda_3)}}
=
0,
\quad
k\in\{0,1\},
\end{equation}
where $n_2$ is the number of times $\lambda_3$ traced the segment $[b_2,b_3]$ along the trajectory.
Since these relations do not depend on the initial point, each trajectory with the same caustics will become closed after $\lambda_1$, $\lambda_2$, $\lambda_3$ traced the corresponding segments $m_1$, $n_1$, $n_2$ times respectively.
\end{proof}

We will denote the underlying hyper-elliptic curve as:
\begin{equation}\label{eq:curve}
 \Curve\ :\ y^2=\Pol(x),
\end{equation}
with $\Pol(x)$ given by \refeq{eq:polynomial}.
The Weierstrass point on $\Curve$ corresponding to the value $x=\xi$, $\xi\in\{\gamma_1,\gamma_2,a_1,a_2,-a_3,\infty\}$ will be denoted by $P_{\xi}$.
One of the points corresponding to $x=0$ will be denoted by $P_0$.

We note that relation \refeq{eq:darboux} implies the following equivalence on the Jacobian of the hyper-elliptic curve $\Curve$:
\begin{equation}\label{eq:divisor}
 m_1(P_0-P_{c_1})+n_1(P_0-P_{b_1})+n_2(P_{b_2}-P_{b_3})\sim0.
\end{equation}

In the following theorem, we present a detailed algebro-geometric characterisation of periodic trajectories whenever the curve $\Curve$ is non-singular.

\begin{theorem}[Algebro-geometric conditions for periodicity]
	\label{th:divisor-conditions}
Consider a billiard trajectory within ellipsoid $\E$ in the Minkowski space $\mathbf{E}^{2,1}$,
with non-degenerate distinct caustics by $\Q_{\gamma_1}$ and $\Q_{\gamma_2}$.
Then the trajectory is $n$-periodic if and only if one of the following is satisfied:
\begin{itemize}
\item The trajectory is space-like, and
\begin{itemize}
\item[(S1)]
both caustics are ellipsoids and either:
\begin{itemize}
\item
$n$ is even, and the divisor $nP_0$ equivalent to one of $nP_{\infty}$, $(n-2)P_{\infty}+P_{\gamma_1}+P_{\gamma_2}$ on the Jacobian of the curve $\Curve$; or
\item
$n$ is odd, and the divisor $nP_0$ equivalent to one of $(n-1)P_{\infty}+P_{\gamma_1}$, $(n-1)P_{\infty}+P_{\gamma_2}$.
\end{itemize}
\item[(S2)]
$\Q_{\gamma_1}$ is ellipsoid, $\Q_{\gamma_2}$ $1$-sheeted hyperboloid along $x_3$-axis, and either:
\begin{itemize}
 \item $n$ is even and $nP_0\sim nP_{\infty}$; or
 \item $n$ is odd and $nP_0\sim(n-1)P_{\infty}+P_{\gamma_1}$.
\end{itemize}
\item[(S3)]
one caustic is a $1$-sheeted hyperboloid oriented along $x_3$-axis, the other a $1$-sheeted hyperboloid oriented along $x_2$-axis, $n$ is even, and $nP_0\sim nP_{\infty}$.
\item[(S4)]
$\Q_{\gamma_1}$ is ellipsoid, $\Q_{\gamma_2}$ $1$-sheeted hyperboloid oriented along $x_2$-axis, and either:
\begin{itemize}
 \item $n$ is even and $nP_0\sim nP_{\infty}$; or
 \item $n$ is odd and $nP_0\sim(n-1)P_{\infty}+P_{\gamma_2}$.
\end{itemize}
\end{itemize}

\item The trajectory is time-like, and
\begin{itemize}
	\item[(T1)]
	$\Q_{\gamma_1}$ is ellipsoid, $\Q_{\gamma_2}$ $1$-sheeted hyperboloid oriented along $x_2$-axis, and either:
	\begin{itemize}
		\item
		$n$ is even and $nP_0\sim nP_{\infty}$; or
		\item
		$n$ is odd and $nP_0\sim(n-1)P_{\infty}+P_{\gamma_1}$.
	\end{itemize}
	\item[(T2)]
	$\Q_{\gamma_1}$ is ellipsoid, $\Q_{\gamma_2}$ $2$-sheeted hyperboloid along $x_3$-axis, and either:
	\begin{itemize}
		\item $n$ is even and $nP_0\sim nP_{\infty}$; or
		\item $n$ is odd and $nP_0\sim(n-1)P_{\infty}+P_{\gamma_1}$.
	\end{itemize}
	\item[(T3)]
	both caustics are $1$-sheeted hyperboloids oriented along $x_2$-axis, $n$ is even, and the divisor $nP_0$ is equivalent to either $nP_{\infty}$ or $(n-2)P_{\infty}+P_{\gamma_1}+P_{\gamma_2}$.
	\item[(T4)]
	one caustic is $1$-sheeted hyperboloid oriented along $x_2$-axis, the other $2$-sheeted hyperboloid oriented along $x_3$-axis, $n$ is even, and $nP_0\sim nP_{\infty}$.
\end{itemize}

\end{itemize}
\end{theorem}
\begin{proof}
For a space-like trajectory, according to Theorem \ref{th:parametri.kaustike}, we have $\gamma_2<0<\gamma_1<a_1$, $\gamma_2\in\{c_1,c_2\}$, $\gamma_1\in\{b_1,b_2\}$.
Thus, there are four possibilities of the types of the caustics.

\emph{Case S1 ($\gamma_2=c_1$, $\gamma_1=b_1$):} Both caustics are ellipsoids, $-a_3<\gamma_2<0<\gamma_1<a_2<a_1$.

The algebro-geometric condition \refeq{eq:divisor} in this case is:
$$
m_1(P_0-P_{\gamma_1})+n_1(P_0-P_{\gamma_2})+n_2(P_{a_2}-P_{a_1})\sim0.
$$
$n_2$ is even, since it is the number of times the particle crossed the plane $x_2=0$ along the closed trajectory, so the condition is equivalent to:
$$
nP_0-m_1P_{\gamma_1}-n_1P_{\gamma_2}\sim0.
$$
From there:
$$
nP_0\sim
\begin{cases}
nP_{\infty}, & \text{if } m_1\text{ and }n_1\text{ are even};
\\
(n-2)P_{\infty}+P_{\gamma_1}+P_{\gamma_2}, & \text{if } m_1\text{ and }n_1\text{ are odd};
\\
(n-1)P_{\infty}+P_{\gamma_1}, & \text{if } m_1\text{ is odd and }n_1\text{ even};
\\
(n-1)P_{\infty}+P_{\gamma_2}, & \text{if } m_1\text{ is even and }n_1\text{ odd}.
\end{cases}
$$

\emph{Case S2 ($\gamma_2=c_2$, $\gamma_1=b_1$):} One caustic is an ellipsoid, and the other $1$-sheeted hyperboloid along $x_3$-axis,
$\gamma_2<-a_3<0<\gamma_1<a_2<a_1$.

The algebro geometric condition \refeq{eq:divisor} for $n$-periodicity is:
$$
m_1(P_0-P_{-a_3})+n_1(P_0-P_{\gamma_1})+n_2(P_{a_2}-P_{a_1})\sim0.
$$
In this case, $m_1$ and $n_2$ must be even, so $n$ and $n_1$ are of the same parity.
Thus
$$
nP_0\sim
\begin{cases}
nP_{\infty}, & \text{if } n\text{ is even};
\\
(n-1)P_{\infty}+P_{\gamma_1}, & \text{if } n\text{ is odd}.
\end{cases}
$$

\emph{Case S3 ($\gamma_2=c_2$, $\gamma_1=b_2$):} One caustic is a $1$-sheeted hyperboloid oriented along $x_3$-axis, and the other a $1$-sheeted hyperboloid oriented along $x_2$-axis:
$\gamma_2<-a_3<0<a_2<\gamma_1<a_1$.

The algebro geometric condition for $n$-periodicity is:
$$
m_1(P_0-P_{-a_3})+n_1(P_0-P_{a_2})+n_2(P_{\gamma_1}-P_{a_1})\sim0,
$$
where $m_1$, $n_1$, $n_2$ are all even, which implies
$nP_0\sim nP_{\infty}$.

\emph{Case S4 ($\gamma_2=c_1$, $\gamma_1=b_2$):} The caustics are a $1$-sheeted hyperboloid oriented along $x_2$-axis and an ellipsoid:
$-a_3<\gamma_2<0<a_2<\gamma_1<a_1$.

The algebro geometric condition for $n$-periodicity is:
$$
m_1(P_0-P_{\gamma_2})+n_1(P_0-P_{a_2})+n_2(P_{\gamma_1}-P_{a_1})\sim0,
$$
with even $n_1$, $n_2$, so $n$ and $m_1$ are of the same parity.
From there we get:
$$
nP_0\sim
\begin{cases}
nP_{\infty},&\text{if }n\text{ is even};
\\
(n-1)P_{\infty}+P_{\gamma_2},&\text{if }n\text{ is odd}.
\end{cases}
$$

For a time-like trajectory, Theorem \ref{th:parametri.kaustike} gives $0<\gamma_1<\gamma_2$,  $\gamma_1\in\{b_1,b_2\}$, $\gamma_2\in\{b_3,b_4\}$.
Again, there are four possibilites for the types of the caustics.

\emph{Case T1 ($\gamma_1=b_1$, $\gamma_2=b_3$):}
One caustic is ellipsoid, the other is $1$-sheeted hyperboloid oriented along $x_2$-axis,
$-a_3<0<\gamma_1<a_2<\gamma_2<a_1$.

The algebro geometric condition for $n$-periodicity is:
$$
m_1(P_0-P_{-a_3})+n_1(P_0-P_{\gamma_1})+n_2(P_{a_2}-P_{\gamma_2})\sim0,
$$
where $m_1$ and $n_2$ must be even, so $n$ and $n_1$ are of the same parity.
Thus this is equivalent to
$nP_0-n_1P_{\gamma_1}-n_2P_{\gamma_2}\sim0$,
i.e.
$$
nP_0\sim
\begin{cases}
nP_{\infty}, &\text{if }n\text{ is even};
\\
(n-1)P_{\infty}+P_{\gamma_1},&\text{if }n\text{ is odd}.
\end{cases}
$$

\emph{Case T2 ($\gamma_1=b_1$, $\gamma_2=b_4$):} The caustics are an ellipsoid and a $2$-sheeted hyperboloid along $x_3$-axis,
$-a_3<0<\gamma_1<a_2<a_1<\gamma_2$.
This case is done identically as Case T1.

\emph{Case T3:} Both caustics are $1$-sheeted hyperboloids along $y$-axis: $\gamma_1=b_2$, $\gamma_2=b_3$.
Here
$-a_3<0<a_2<\gamma_1<\gamma_2<a_1$.

The algebro geometric condition for $n$-periodicity is:
$$
m_1(P_0-P_{-a_3})+n_1(P_0-P_{a_2})+n_2(P_{\gamma_1}-P_{\gamma_2})\sim0,
$$
where $n_1$, $m_1$ are both even, so $n$ is also even.
We get:
$$
nP_0\sim
\begin{cases}
 nP_{\infty}, &\text{if }n_2\text{ is even};
 \\
 (n-2)P_{\infty}+P_{\gamma_1}+P_{\gamma_2}, &\text{if }n_2\text{ is odd}.
 \end{cases}
$$

\emph{Case T4 ($\gamma_1=b_2$, $\gamma_2=b_4$):} The caustics are a $1$-sheeted hyperboloid along $x_2$-axis and a $2$-sheeted hyperboloid along $x_3$-axis,
$-a_3<0<a_2<\gamma_1<a_1<\gamma_2$.

The algebro geometric condition for $n$-periodicity is:
$$
m_1(P_0-P_{-a_3})+n_1(P_0-P_{a_2})+n_2(P_{\gamma_1}-P_{a_1})\sim0,
$$
where $m_1$, $n_1$, $n_2$ are all even, which than gives $nP_0\sim nP_{\infty}$.
\end{proof}

The analytic Cayley-type conditions for periodic trajectories can be derived from Theorem \ref{th:divisor-conditions} using the next Lemma.

\begin{lemma}\label{lemma:cayley}
Consider a non-singular curve $\Curve$ \refeq{eq:curve}.
Then:
\begin{itemize}
	\item $nP_{0}\sim nP_{\infty}$ for $n$ even if and only if $n\ge6$ and
	$$
	\rank\left(
	\begin{array}{llll}
	A_4 & A_5 &\dots & A_{m+1}\\
	A_5 & A_6 &\dots & A_{m+2}\\
	\dots\\
	A_{m+2}& A_{m+3}&\dots & A_{2m-1}
	\end{array}
	\right)
	<m-2,
	\quad n=2m,
	$$
with $\sqrt{\Pol(x)}=A_0+A_1x+A_2x^2+\dots$;

\item
$nP_0\sim (n-2)P_{\infty}+P_{\gamma_1}+P_{\gamma_2}$ for $n$ even if and only if $n\ge4$ and
$$
\rank\left(
\begin{array}{llll}
B_2 & B_3 &\dots & B_{m}\\
B_3 & B_4 &\dots & B_{m+1}\\
\dots\\
B_{m+1}& B_{m+2}&\dots & B_{2m-1}
\end{array}
\right)
<m-1,
\quad n=2m,
$$
with $\dfrac{\sqrt{\Pol(x)}}{(x-\gamma_1)(x-\gamma_2)}=B_0+B_1x+B_2x^2+\dots$;

\item
$nP_0\sim (n-1)P_{\infty}+P_{\gamma_1}$ for $n$ odd if and only if $n\ge5$ and
$$
\rank\left(
\begin{array}{llll}
C_3 & C_4 &\dots & C_{m+1}\\
C_4 & C_5 &\dots & C_{m+2}\\
\dots\\
C_{m+2}& C_{m+3}&\dots & C_{2m}
\end{array}
\right)
<m-1,
\quad n=2m+1,
$$
with $\dfrac{\sqrt{\Pol(x)}}{x-\gamma_1}=C_0+C_1x+C_2x^2+\dots$;

\item
$nP_0\sim (n-1)P_{\infty}+P_{\gamma_2}$ for $n$ odd if and only if $n\ge5$ and
$$
\rank\left(
\begin{array}{llll}
D_3 & D_4 &\dots & D_{m+1}\\
D_4 & D_5 &\dots & D_{m+2}\\
\dots\\
D_{m+2}& D_{m+3}&\dots & D_{2m}
\end{array}
\right)
<m-1,
\quad n=2m+1,
$$
with $\dfrac{\sqrt{\Pol(x)}}{x-\gamma_2}=D_0+D_1x+D_2x^2+\dots$.

\end{itemize}
\end{lemma}
\begin{proof}
When $n=2m$ is even, the basis for $\LL(nP_{\infty})$ is:
$$\{1,x,x^2,\dots, x^m, y, xy, \dots, x^{m-3}y\},$$
while
$\LL((n-2)P_{\infty}+P_{\gamma_1}+P_{\gamma_2})$ has basis
$$
\left\{1,x,x^2,\dots,x^{m-1},\frac{y}{(x-\gamma_1)(x-\gamma_2)},
\frac{xy}{(x-\gamma_1)(x-\gamma_2)},
\dots,
\frac{x^{m-2}y}{(x-\gamma_1)(x-\gamma_2)}
\right\}.
$$
When $n=2m+1$, the basis for $\LL\left((n-1)P_{\infty}+P_{\gamma_1}\right)$ is:
$$
\left\{1,x,x^2,\dots,x^m,\frac{y}{x-\gamma_1},
\frac{xy}{x-\gamma_1},
\dots,
\frac{x^{m-2}y}{x-\gamma_1}
\right\}.
$$

In each case, the condition for the divisors equivalence is that there is a linear combination of the basis with a zero of order $n$ at $x=0$, which gives $n$ linear equations for the coefficients.
In order to get a non-trivial solutions, the rank of the system cannot be maximal, which gives the stated conditions, as it was done in \cites{GrifHar1978,DragRadn2011book}.
\end{proof}

Next, we will consider the case when the two caustics coincide: $\gamma_1=\gamma_2$.
Then, the segments of a billiard trajectory within $\E$ are generatrices of the double caustic, which must be a $1$-sheeted hyperboloid oriented along $x_2$-axis.
Such a situation can be considered as a limit of the case T3 from Theorem \ref{th:divisor-conditions}, when $\gamma_2\to\gamma_1$.
The Cayley-type condition for periodicity is thus obtained by taking the limit of the correspondinc analytic condition from Lemma \ref{lemma:cayley}.

\begin{proposition}\label{prop:cayley-double-caustic}
A billiard trajectory within $\E$ with segments on $1$-sheeted hyperboloid $\Q_{\gamma_1}$, which is oriented along $x_2$-axis, is $n$ periodic if and only if $n$ is even and either
\begin{itemize}
	\item
$$
\rank\left(
\begin{array}{llll}
A_4 & A_5 &\dots & A_{m+1}\\
A_5 & A_6 &\dots & A_{m+2}\\
\dots\\
A_{m+2}& A_{m+3}&\dots & A_{2m-1}
\end{array}
\right)
<m-2,
\quad n=2m\ge6,
$$
with
$(\gamma_1-x)\sqrt{(a_1-x)(a_2-x)(a_3+x)}=A_0+A_1x+A_2x^2+\dots$; or

\item
$$
\rank\left(
\begin{array}{llll}
B_2 & B_3 &\dots & B_{m}\\
B_3 & B_4 &\dots & B_{m+1}\\
\dots\\
B_{m+1}& B_{m+2}&\dots & B_{2m-1}
\end{array}
\right)
<m-1,
\quad n=2m\ge4,
$$
with $\dfrac{\sqrt{(a_1-x)(a_2-x)(a_3+x)}}{\gamma_1-x}=B_0+B_1x+B_2x^2+\dots$.
\end{itemize}
\end{proposition}

Finally, we will consider light-like trajectories.
Such trajectories can be considered as a limit of Cases S2 and S3 from Theorem \ref{th:divisor-conditions}, when $\gamma_2\to-\infty$, or a limit of Cases T2 and T4, with $\gamma_2\to+\infty$.
The analytic conditions are obtained as the limit of the corresponding conditions from Lemma \ref{lemma:cayley}.

\begin{proposition}\label{prop:cayley-light-like}
A light-like billiard trajectory within $\E$, with non-degenerate caustic $\Q_{\gamma_1}$, is $n$-periodic if and only if
\begin{itemize}
	\item $n$ is even, $n\ge6$, and
		$$
	\rank\left(
	\begin{array}{llll}
	A_4 & A_5 &\dots & A_{m+1}\\
	A_5 & A_6 &\dots & A_{m+2}\\
	\dots\\
	A_{m+2}& A_{m+3}&\dots & A_{2m-1}
	\end{array}
	\right)
	<m-2,
	\quad n=2m,
	$$
	with
	$\sqrt{(a_1-x)(a_2-x)(a_3+x)(\gamma_1-x)}=A_0+A_1x+A_2x^2+\dots$; or	
\item $\Q_{\gamma_1}$ is an ellipsoid, $n$ is odd, $n\ge5$, and
$$
\rank\left(
\begin{array}{llll}
B_3 & B_4 &\dots & B_{m+1}\\
B_4 & B_5 &\dots & B_{m+2}\\
\dots\\
B_{m+2}& B_{m+3}&\dots & B_{2m}
\end{array}
\right)
<m-1,
\quad n=2m+1,
$$
with $\sqrt{\dfrac{{(a_1-x)(a_2-x)(a_3+x)}}{\gamma_1-x}}=B_0+B_1x+B_2x^2+\dots$.\end{itemize}
\end{proposition}

\section{Polynomial equations}\label{sec:polynomial}

In this section, we express the periodicity conditions as polynomial functional equations.

\begin{lemma}\label{lemma:polynomial}
Consider a non-singular curve $\Curve$ \refeq{eq:curve}. Then:
\begin{itemize}
\item $nP_{0}\sim nP_{\infty}$ for $n=2m$ if and only if $n\ge6$ and there are real polynomials $p_m(s)$ and $q_{m-3}(s)$ of degrees $m$ and $m-3$ respectively such that
$$
p_m^2(s)
-s\left(s-\frac1{a_1}\right)\left(s-\frac1{a_2}\right)\left(s+\frac1{a_3}\right)
\left(s-\frac1{\gamma_1}\right)\left(s-\frac1{\gamma_2}\right)q_{m-3}^2(s)=1;
$$
		
\item
		$nP_0\sim (n-2)P_{\infty}+P_{\gamma_1}+P_{\gamma_2}$ for $n=2m$ even if and only if $n\ge4$ and there are real polynomials $p_{m-1}(s)$ and $q_{m-2}(s)$ of degrees $m-1$ and $m-2$ respectively such that
$$
\left(s-\frac1{\gamma_1}\right)\left(s-\frac1{\gamma_2}\right)p_{m-1}^2(s)
-s\left(s-\frac1{a_1}\right)\left(s-\frac1{a_2}\right)\left(s+\frac1{a_3}\right)
q_{m-2}^2(s)=\varepsilon,
$$
with $\varepsilon=\sign(\gamma_1\gamma_2)$;
		
		\item
		$nP_0\sim (n-1)P_{\infty}+P_{\gamma_1}$ for $n=2m+1$ odd and $\gamma_1>0$ if and only if $n\ge5$ and there are real polynomials $p_{m}(s)$ and $q_{m-2}(s)$ of degrees $m$ and $m-2$ respectively such that
		$$
		\left(s-\frac1{\gamma_1}\right)p_{m}^2(s)
		-s\left(s-\frac1{a_1}\right)\left(s-\frac1{a_2}\right)\left(s+\frac1{a_3}\right)
		\left(s-\frac1{\gamma_2}\right)
		q_{m-2}^2(s)=-1;
		$$		
	
	\item
		$nP_0\sim (n-1)P_{\infty}+P_{\gamma_2}$ for $n=2m+1$ odd and $\gamma_2<0$ if and only if $n\ge5$ and there are real polynomials $p_{m}(s)$ and $q_{m-2}(s)$ of degrees $m$ and $m-2$ respectively such that
		$$
		\left(s-\frac1{\gamma_2}\right)p_{m}^2(s)
		-s\left(s-\frac1{a_1}\right)\left(s-\frac1{a_2}\right)\left(s+\frac1{a_3}\right)
		\left(s-\frac1{\gamma_1}\right)
		q_{m-2}^2(s)=1.
		$$		
\end{itemize}
\end{lemma}
\begin{proof}
It is clear from the proof of Lemma \ref{lemma:cayley} that the relation
$2mP_0\sim 2mP_{\infty}$ is satisfied if and only if there are real polynomials $p_m^*(x)$ and $q_{m-3}^*(x)$ such that $p_m^*(x)+q_{m-3}^*(x)\sqrt{\Pol(x)}$ has a zero of multiplicity $2m$ at $x=0$.
Multiplying that expression by $p_m^*(x)-q_{m-3}^*(x)\sqrt{\Pol(x)}$, we get that the polynomial $(p_m^*(x))^2-\Pol(x)(q_{m-3}^*(x))^2$, which is of degree $2m$, has a zero of order $2m$ at $x=0$.
Assuming that $p_m^*$ is monic, we have:
$$
(p_m^*(x))^2-\Pol(x)(q_{m-3}^*(x))^2=x^{2m}.
$$
Dividing by $x^{2m}$ and introducing $s=1/x$, we get the needed relation.

The relation $2mP_0\sim (2m-2)P_{\infty}+P_{\gamma_1}+P_{\gamma_2}$ is satisfied if and only if there are real polynomials $p_{m-1}^*(x)$ and $q_{m-2}^*(x)$ of degrees $m-1$ and $m-2$ such that
$$
p_{m-1}^*(x)+q_{m-2}^*(x)\frac{\sqrt{\Pol(x)}}{(\gamma_1-x)(\gamma_2-x)}
$$
has a zero of order $2m$ at $x=0$.
Multyplying by:
$$
\varepsilon(\gamma_1-x)(\gamma_2-x)
\left(p_{m-1}^*(x)-q_{m-2}^*(x)\frac{\sqrt{\Pol(x)}}{(\gamma_1-x)(\gamma_2-x)}\right),
$$
we get that the polynomial:
$$
\varepsilon(\gamma_1-x)(\gamma_2-x)(p_{m-1}^*(x))^2-(a_1-x)(a_2-x)(a_3+x)(q_{m-2}^*(x))^2,
$$
which is of degree $2m$,
has a zero of order $2m$ at $x=0$.
Thus, it equals $\varepsilon x^{2m}$.
Dividing by $x^{2m}$ and introducing $s=1/x$, we get the stated polynomial relation.

The relation $(2m+1)P_0\sim 2mP_{\infty}+P_{\gamma_1}$ is satisfied if and only if there are real polynomials $p_{m}^*(x)$ and $q_{m-2}^*(x)$ of degrees $m$ and $m-2$ respectively such that
$$
p_{m}^*(x)+q_{m-2}^*(x)\frac{\sqrt{\Pol(x)}}{\gamma_1-x}
$$
has a zero of order $2m+1$ at $x=0$.
Multyplying by:
$$
(\gamma_1-x)
\left(p_{m}^*(x)-q_{m-2}^*(x)\frac{\sqrt{\Pol(x)}}{\gamma_1-x}\right),
$$
we get that the polynomial:
$$
(\gamma_1-x)(p_{m}^*(x))^2-\varepsilon(a_1-x)(a_2-x)(a_3+x)(\gamma_2-x)(q_{m-2}^*(x))^2,
$$
which is of degree $2m+1$,
has a zero of order $2m+1$ at $x=0$.
Assuming that $p_m^*(x)$ is monic, we get that the last expression equals $-x^{2m+1}$.
Note that $\varepsilon=\sign(\gamma_1\gamma_2)=\sign(\gamma_2)$.
Dividing by $x^{2m+1}$, introducing $s=1/x$, we get the stated relation.

The relation $(2m+1)P_0\sim 2mP_{\infty}+P_{\gamma_2}$ is satisfied if and only if there are real polynomials $p_{m}^*(x)$ and $q_{m-2}^*(x)$ of degrees $m$ and $m-2$ respectively such that
$$
p_{m}^*(x)+q_{m-2}^*(x)\frac{\sqrt{\Pol(x)}}{\gamma_2-x}
$$
has a zero of order $2m+1$ at $x=0$.
Multyplying by:
$$
(\gamma_2-x)
\left(p_{m}^*(x)-q_{m-2}^*(x)\frac{\sqrt{\Pol(x)}}{\gamma_2-x}\right),
$$
we get that the polynomial:
$$
(\gamma_2-x)(p_{m}^*(x))^2-\varepsilon(a_1-x)(a_2-x)(a_3+x)(\gamma_1-x)(q_{m-2}^*(x))^2,
$$
which is of degree $2m+1$,
has a zero of order $2m+1$ at $x=0$.
Assuming that $p_m^*(x)$ is monic, we get that the last expression equals $-x^{2m+1}$.
Note that $\varepsilon=\sign(\gamma_1\gamma_2)=-\sign(\gamma_1)$.
Dividing by $-x^{2m+1}$, introducing $s=1/x$, we get the stated relation.
\end{proof}

By taking the appropriate limits, we get the polynomial conditions for the case of a double caustic and the case of light-like trajectories:
\begin{proposition}\label{prop:polynomial-singular}
(a)
A billiard trajectory within $\E$ with segments on $1$-sheeted hyperboloid $\Q_{\gamma_1}$, which is oriented along $x_2$-axis, is $n$ periodic if and only if $n=2m$ is even and either:
\begin{itemize}
	\item $n\ge6$ and there are real polynomials $p_m(s)$ and $q_{m-3}(s)$ of degrees $m$ and $m-3$ respectively such that
	$$
	p_m^2(s)
	-s\left(s-\frac1{a_1}\right)\left(s-\frac1{a_2}\right)\left(s+\frac1{a_3}\right)
	\left(s-\frac1{\gamma_1}\right)^2q_{m-3}^2(s)=1;
	$$
	
	\item
 $n\ge4$ and there are real polynomials $p_{m-1}(s)$ and $q_{m-2}(s)$ of degrees $m-1$ and $m-2$ respectively such that
	$$
	\left(s-\frac1{\gamma_1}\right)^2p_{m-1}^2(s)
	-s\left(s-\frac1{a_1}\right)\left(s-\frac1{a_2}\right)\left(s+\frac1{a_3}\right)
	q_{m-2}^2(s)=1.
	$$
\end{itemize}
(b)
A light-like billiard trajectory within $\E$, with non-degenerate caustic $\Q_{\gamma_1}$, is $n$-periodic if and only if
\begin{itemize}
	\item $n=2m$ is even, $n\ge6$, $\Q_{\gamma_1}$ is an ellipsoid or a $1$-sheeted hyperboloid oriented along $x_2$-axis, and there are real polynomials $p_m(s)$ and $q_{m-3}(s)$ of degrees $m$ and $m-3$ respectively such that
	$$
	p_m^2(s)
	-s^2\left(s-\frac1{a_1}\right)\left(s-\frac1{a_2}\right)\left(s+\frac1{a_3}\right)
	\left(s-\frac1{\gamma_1}\right)q_{m-3}^2(s)=1;
	$$
\item $n=2m+1$ is odd, $n\ge5$, $\Q_{\gamma_1}$ is an ellipsoid, and there are real polynomials $p_{m}(s)$ and $q_{m-2}(s)$ of degrees $m$ and $m-2$ respectively such that
$$
\left(s-\frac1{\gamma_1}\right)p_{m}^2(s)
-s^2\left(s-\frac1{a_1}\right)\left(s-\frac1{a_2}\right)\left(s+\frac1{a_3}\right)
q_{m-2}^2(s)=-1.
$$		

\end{itemize}

\end{proposition}
	
\begin{corollary}\label{cor:pell}
If the billiard trajectories within $\E$ with caustics $\Q_{\gamma_1}$ and $\Q_{\gamma_2}$ are $n$-periodic, then there exist real polynomials $\hat{p}_n$ and $\hat{q}_{n-3}$ of degrees $n$ and $n-3$ respectively, which satisfy the Pell equation:
$$
\hat{p}_n^2(s)-s\left(s-\frac1{a_1}\right)\left(s-\frac1{a_2}\right)\left(s+\frac1{a_3}\right)
\left(s-\frac1{\gamma_1}\right)
\left(s-\frac1{\gamma_2}\right)
\hat q_{n-3}^2(s)=1.
$$
\end{corollary}
\begin{proof}
If $n=2m$, we know that one of the first two cases of Lemma \ref{lemma:polynomial} is satisfied.
In the first case, take $\hat{p}_n=2p_{m}^2-1$ and $\hat{q}_{n-3}=2p_mq_{m-3}$.
In the second case, we set:
$$
\hat{p}_n(s)=2\left(s-\frac1{\gamma_1}\right)\left(s-\frac1{\gamma_2}\right)p_{m-1}^2(s)
-\varepsilon,
\quad
\hat{q}_{n-3}=2p_{m-1}q_{m-2}.
$$
If $n=2m+1$, one of the last two cases of Lemma \ref{lemma:polynomial} holds.
In the third case, we set:
$$
\hat{p}_n(s)=2\left(s-\frac1{\gamma_1}\right)p_{m}^2(s)+1,
\quad
\hat{q}_{n-3}=2p_{m}q_{m-2},
$$
and in the fourth one:
$$
\hat{p}_n(s)=2\left(s-\frac1{\gamma_2}\right)p_{m}^2(s)-1,
\quad
\hat{q}_{n-3}=2p_{m}q_{m-2}.
$$
\end{proof}

\begin{remark}\label{rem:open-problem}
By considering light-like trajectories with an ellipsoid as caustic, and taking the limit when parameter of the caustic approaches zero, we get the light-like geodesics on the ellipsoid $\E$.
Applying the appropriate limit to the analytic conditions for periodicity obtained in this work, the conditions for periodicity for the Poncelet-style closure theorem for light-like geodesics in the equatorial belt from \cite{GKT2007} can be obtained, thus solving the Problem 5.2 from that paper, see also Problem 7 from \cite{Tab2015}.
\end{remark}

\subsection*{Acknowledgment}
The research of V.~D.~and M.~R.~was supported by the Discovery Project \#DP190101838 \emph{Billiards within confocal quadrics and beyond} from the Australian Research Council and Project \#174020 \emph{Geometry and Topology of Manifolds, Classical Mechanics and Integrable Systems} of the Serbian Ministry of Education, Technological Development and Science.
The authors are grateful to the referee for careful reading and very useful comments and suggestions.

\end{document}